\documentclass[11pt, one side,emlines]{amsart}
\usepackage{amssymb,latexsym,xy,eucal,mathrsfs}
\usepackage{MnSymbol}
\textwidth=15.2cm \textheight=24cm \theoremstyle{plain}
\newtheorem{lemma}{Lemma}[section]
\newtheorem{theorem}[lemma]{Theorem}
\newtheorem{proposition}[lemma]{Proposition}
\newtheorem{corollary}[lemma]{Corollary}

\theoremstyle{definition}
\newtheorem{example}[lemma]{Example}

\numberwithin{equation}{section}  \voffset
-55truept \hoffset -15truept
\begin{document}
\baselineskip 15truept
\title{On Unitification of $*$-rings}

\date{}

\author{ Sanjay More, Anil Khairnar and B. N. Waphare}
 \address {\rm Department of Mathematics, Anantrao Pawar College, Pune-412115, India.}
	\email{\emph{sanjaymore71@gmail.com}} 
	 
\address{\rm Department of Mathematics, Abasaheb Garware College, Pune-411004, India.}
 \email{\emph{ask.agc@mespune.in; anil\_maths2004@yahoo.com}}
 
\address{\rm Center For Advanced Studies in Mathematics, Department of Mathematics, 
Savitribai Phule Pune University, Pune-411007, India.}
 \email{\emph{bnwaph@math.unipune.ac.in; waphare@yahoo.com}}
 
 
\maketitle {\bf \small Abstract:} {\small 
	S. K. Berberian raised the open problem ``Can every weakly Rickart $*$-ring be
	embedded in a Rickart $*$-ring? with preservation of right projections?" Berberian has given a partial solution to this problem. Khairnar and Waphare raised a similar problem for p.q.-Baer $*$-rings and gave a partial solution. In this paper, we give more general partial solutions to both the problems.  
 }
\section{Introduction} 
\indent   Kaplansky \cite{Kap} introduced Baer rings and Baer $*$-rings to
abstract various properties of $AW^*$-algebras (that is a $C^*$-algebra which is also a Baer $*$-ring), von Neumann
algebras and complete $*$-regular rings. The concept of a Baer $*$-ring is naturally motivated from the study of functional analysis. For example, every von Neumann algebra is a Baer $*$-algebra. One can refer \cite{Anil2,Anil3,Anil4,Lia2,Lia3, Lia} for recent work on rings with involution.

     \indent Throughout this paper, $R$ denotes an associative ring. An {\it ideal} of a ring $R$, 
    we mean a two sided ideal. A ring $R$ is said to be {\it reduced} if it does not 
    have a nonzero nilpotent element.  A ring $R$ is said to be {\it abelian} if its every idempotent element is central.  
   Let $S$ be a nonempty subset of $R$. We write
          $r_R(S)= \{a \in R ~|~ s a = 0, ~\forall ~s \in S \}$, and is called the
          {\it right annihilator} of $S$ in $R$, and $l_R(S)= \{a \in R ~|~  a s = 0, ~\forall ~s \in S \}$, 
          is the
         { \it left annihilator} of $S$ in $R$.         
    A {\it $*$-ring $R$} is a ring equipped with an involution $x \rightarrow x^* $,
                        that is, an additive anti-automorphism of the period at most two.
    An element $e$ of a $*$-ring  $R$ is called a {\it projection} if it is self-adjoint 
         (i.e. $e=e^*$)
       and idempotent (i.e. $e^2=e$). 
      A $*$-ring $R$ is said to be a {\it  Rickart $*$-ring},
  if for each $x \in R$, $r_R(\{x\})=eR$, where $e$ is a projection
  in $R$. 
   For each element $a$ in a Rickart $*$-ring, there is unique projection
  $e$ such that $ae=a$ and $ax=0$ if and only if $ex=0$, called the {\it
  	right projection} of $a$ denoted by $RP(a)$. Similarly, the left
  projection $LP(a)$ is defined for each element $a$ in Rickart $*$-ring.
  A $*$-ring $R$ is said to be a {\it weakly Rickart  $*$-ring}, if for any $x \in R$, there exists a projection $e$
  such that (1) $xe=x$, and (2) if $xy=0$ then $ey=0$.\\
    Recall the following propositions and an open problem from \cite{Ber}.
    \begin{proposition}[{\cite[Proposition 2, page 13]{Ber}}] \label{Pro1}  If $R$ is Rickart $*$-ring, then $R$ has a unity element and the involution of $R$ is proper.        	
 \end{proposition} 	
 \begin{proposition}  [{\cite[Proposition 2, page 28]{Ber}}]\label{Pro2}
 	The following condition on a $*$-ring $R$ are equivalent:\\
 	(a)$R$ is a Rickart $*$-ring;\\
 	(b)$R$  is weakly Rickart $*$-ring with unity. 
 \end{proposition}       
 Proposition \ref{Pro1} says that the unity element exist in any Rickart $*$-ring and the proposition \ref{Pro2} naturally create the following problem.\\
 {\bf Problem 1:} Can every weakly Rickart $*$-ring be
 embedded in a Rickart $*$-ring with preservation of $RP$'s?\\
 In \cite{Ber} Berberian has given a partial solution to this problem. 
 
  \indent  According to Birkenmeier et al. \cite{Bir}, a $*$-ring $R$ is said to be a {\it quasi-Baer $*$-ring} if the right annihilator of every
  ideal of $R$ is generated, as a right ideal, by a projection in $R$.        
  Birkenmeier et al. \cite{Bir5} introduced principally quasi-Baer (p.q.-Baer) $*$-rings as a generalization of quasi-Baer $*$-rings.
  A $*$-ring $R$ is said to be a {\it p.q.-Baer $*$-ring}, if for every principal right ideal $aR$ of $R$, $r_R(aR)=eR$, where $e$ is a projection in
    $R$, it follows that $l_R(Ra)=Rf$ for
       a suitable projection $f$. Note that an abelian Rickart $*$-ring is a p.q.-Baer $*$-ring, and a reduced p.q.-Baer $*$-ring is a Rickart $*$-ring.   
   We say that an element $x$ of a $*$-ring $R$
  possesses a {\it central cover} if there exists a smallest central
  projection $h \in R$ such that $hx=x$. If such a projection $h$ exists,
  then it is unique, it is called the central cover of $x$, denoted
  by $h=C(x)$.   
 In \cite{Anil} Khairnar  and Waphare proved that the central cover exists for every element in any p.q.- Baer$*$-ring.        
\begin{theorem}[{\cite[Theorem 2.5]{Anil}}] Let $R$ be a p.q.- Baer $*$-ring and $x\in R$.
	Then $x$ has a central cover $e \in R$. Further,
	$xRy=0$ if and only if $yRx=0$ if and only if $ey=0$.\\
	That is $r_R(xR)=r_R(eR)=l_R(Rx)=l_R(Re)=(1-e)R=R(1-e)$.
\end{theorem}      
In \cite{Anil} Khairnar  and Waphare introduced the concept of weakly p.q.- Baer $*$-ring. A $*$-ring $R$ is said to be a {\it weakly
	p.q.-Baer $*$-ring}, if every $x\in R$ has a central cover $e \in
R$ such that, $xRy=0$  if and only if $ey=0$. According to \cite{Bir5}, the involution $*$ of a $*$-ring
$R$ is {\it semi-proper}, if for any $a\in R$, $aRa^*= 0$ implies $a=0$.\\
 Recall the following results and an open problem from \cite{Anil}. 
\begin{proposition}[{\cite[Proposition 2.4]{Anil}}]
	If $R$ is p.q.Baer $*$-ring, then $R$ has the unity element and the involution of $R$ is semi-proper.
\end{proposition}         
\begin{theorem} [{\cite[Theorem 3.9]{Anil}}] The following conditions on a
	$*$-ring $R$ are equivalent:\\
	(a) $R$ is a p.q.-Baer $*$-ring.\\
	(b) $R$ is a weakly p.q.-Baer $*$-ring with unity.
\end{theorem}  
In view of the above theorem, the following problem is raised in \cite{Anil}.\\
{\bf Problem 2:} Can every weakly p.q.-Baer $*$-ring be
embedded in a p.q.-Baer $*$-ring? with preservation of central covers?\\        
\indent  In \cite{Anil}, Khairnar and Waphare provided a partial solution to Problem 2. \\
     \indent  In the second section of this paper, we give a more general partial solution of problem 1 and  in section 3, we give a more general partial solution of problem 2. 
  
 \section{Unitification of Weakly Rickart $*$-rings}    
 
 Recall the definition of unitification of a $*$-ring given by Berberian \cite{Ber}.
 Let $R$ be a $*$-ring. We say that $R_1$ is a unitification of $R$, if there exists a  ring $K$, such that,\\
 1) $K$ is an integral domain with involution (necessarily
 proper), that is, $K$ is a commutative $*$-ring with unity
 and without divisors of zero (the identity involution is
 permitted),\\
 2) $R$ is a $*$-algebra over $K$ (that is, $R$ is a left $K$-module such that,
 identically
 $1a=a,~\lambda (ab)=(\lambda a)b=a (\lambda b),~ and~(\lambda a)^*=\lambda ^*
 a^*$ for $\lambda \in K$ and $a,b \in R$).\\
 3) $R$ is torsion free $K$-module (that is $\lambda a=0$ implies $\lambda =0 $ or $a=0$).\\
 Define $R_1=R \oplus K$ (the additive group direct sum), thus
 $(a, \lambda)=(b, \mu)$ means, by the definition that $a=b$ and $\lambda
 =\mu$, and addition in $R_1$, is defined by the formula $(a, \lambda)+(b, \mu)=(a+b, \lambda +
 \mu)$. Define  $(a, \lambda)(b, \mu)=(ab+ \mu a+ \lambda b, \lambda
 \mu)$, $\mu (a, \lambda)=(\mu a, \mu \lambda)$, $(a, \lambda)^*=(a^*, \lambda
 ^*)$. Evidently, $R_1$ is also a $*$-algebra over $K$, has unity
 element $(0, 1)$ and $R$ is a
 $*$-ideal in $R_1$.\\ 
  The following lemmas are elementary facts about unitification $R_1$ of a $*$-ring $R$.
 \begin{lemma} [{\cite[Lemma 1, page 30]{Ber}}] \label{tc4s4lm1} With notations as in the definition
 	of unitification, if an involution on $R$ is proper, then so is the
 	involution of $R_1$.
 \end{lemma}
\begin{lemma} [{\cite[Lemma 3, page 30]{Ber}}] \label{tc4s4lm2} With notations as in the definition
	of unitification, let $x \in R$ and let $e$ be a projection in $R$. Then $RP(x)=e$
	in $R$ if and only if $RP((x,0))=(e,0)$ in
	$R_1$.
\end{lemma} 
   
  Berberian has given a partial solution to Problem 1 as follows.
\begin{theorem}[{\cite[Theorem 1, page 31]{Ber}}] \label{Th1}
	Let $R$ be a weakly Rickart $*$-ring. If there exists an involutory integral domain $K$ such that $R$ is a $*$-algebra over $K$  and it is a torsion-free $K$-module, then $R$ can be embedded in a Rickart $*$-ring with preservation of RP's. 
\end{theorem} 

After 1972, there was not much headway towards the solution of Problem 1. In 1996 Thakare and Waphare supplied partial solutions wherein the condition on the underlying weakly Rickart $*$-rings was weakened in two distinct ways. For the solution of this open problem, Berberian used the condition that $R$ is torsion free left $K$-module, and $K$ is an integral domain. Thakare and Waphare gave another solution in which the condition of torsion free is replaced by other condition. They establish the following. 
\begin{theorem}[{\cite[Theorem 2]{Tha3}}] \label{Th2}
	A weakly Rickart $*$-ring $R$ can be embedded into a Rickart $*$-ring, provided there exists a ring $K$ such that
	\begin{enumerate}
\item $K$ is an integral domain with involution,
	\item  $R$ is $*$-algebra over $K$, and
	\item  For any $\lambda \in K-\{0\}$, there exist a projection $e_\lambda$ that is an upper bound for the set of left projections of the right annihilators of $\lambda$, that is if $x\in R$ and $\lambda x=0$ then $LP(x)\leq e_\lambda$.
		\end{enumerate}
\end{theorem}
\begin{theorem}[{\cite[Theorem 7]{Tha3}}]
	$A$ weakly Rickart $*$-ring $R$ can be embedded into Rickart $*$-ring provided the characteristic of $R$ is non zero.
\end{theorem}
The $*$-ring $ C_\infty(T) \bigoplus M_2(\mathbb{Z}_3)$ has no embedding in the sense of Theorem \ref{Th1} as the characteristic of $R$ is zero though it has unitification in the sense of Theorem \ref{Th2}. This is the example that shows that Theorem \ref{Th2} is an improvement over Theorem \ref{Th1} of Berberian.

   Now we prove the existence of largest projection corresponding to the self adjoin element by using condition (3) of the above theorem.
   
 \begin{lemma} \label{ps2L3}	
	  Let $R$ be a weakly Rickart $*$-ring with condition $(3)$ in Theorem \ref{Th2}. Then for any self-adjoint element $a$ and $ \lambda \neq 0 $ there exists a largest projection $g$ such that $ ag = \lambda g $. 
 \end{lemma} 
\begin{proof}  Let $RP(a)=e' $ and $ e_{\lambda} $ be the projection as given by condition $(3)$ of Theorem \ref{Th2}. Let $e=e' \vee e_{\lambda}$, then $e' \leq e$ and $e'=e'e=ee'$. Since $ae'=a$, therefore $ae'e=ae$. Hence $a=ae'=ae$. Also, $a*=a$ implies that $a=ea=eae \in eRe$. Thus $a- \lambda e \in eRe$. Let $h=RP(a- \lambda e)$ and $g=e -h$. This gives $(a- \lambda e) g =0$, hence $ag= \lambda g$. Let $k$ be any projection in $R$ such that $a k = \lambda k$. Consider $ \lambda (ek-k)=e \lambda k - \lambda k= e ak - \lambda k = a k- \lambda k =0$.  Let $LP(e k -k)= f$. Therefore $f \leq e_{\lambda} \leq e$. That is $ek -k= f(ek-k)=fe(ek-k)=f(ek-ek)=0$. Consider $(a- \lambda e) k = ak -\lambda ek = ak - \lambda k =0$. Therefore $RP(a- \lambda e)k = hk=0$. Hence $kg=k(e-h)= ke -kh=k-0=k$. That is $k \leq g$. Therefore $g$ is largest projection such that $ag= \lambda g$.	
\end{proof}	

Recall the following lemma from \cite{Ber}.

\begin{lemma} [{\cite[Lemma 5, page 31]{Ber}}] \label{ps2L4}  
	Let $B$ be a $*$-ring with proper involution, $x \in B$  and $e$ be a  projection in $B$. Then $e$  is the right projection of $x$ if and only if $e$ is the right projection of $x^*x$.
\end{lemma}	 

	We give a solution of Problem 1 in which the condition ``$K$ is an integral domain'' is replaced by ``$K$ is a commutative ring with unity''.
	
	Let $R$ be a $*$-ring, $K$ be a commutative $*$-ring with unity and $R$ be an algebra over $K$. Write $r=r(R,+)$ for the endomorphism ring of the additive group of $R$. Each $ a \in R $ determines an element $L_a$ of $r$ via $ L_a  x = a x $ and each $ \lambda \in  K $ an element $ \lambda I $ of $ r$ via $(\lambda I) x = \lambda x $.
	Let $ R_1 = R \oplus K $ with the $*$-algebra operations as follows $ ( a, \lambda ) + (b, \mu ) = ( a+b , \lambda + \mu ), ~\mu ( a , \lambda) = ( \mu a, \mu  \lambda ) $, $( a , \lambda)( b , \mu ) = ( ab + \mu  a + \lambda  b ,  \lambda ~\mu) $, $(a, \lambda)^* = ( a^* , \lambda^* ) $. Each $ ( a , \lambda )  \in R_1 $  determines an element $ L_a + \lambda I $ of $ r $ and the mapping $ ( a, \lambda) \rightarrow L_a  + \lambda I $ is ring homomorphism of $ R_1 $ onto  a sub-ring  $ S$ of $r$ namely the subring of $r$ generated by $ L_a $  and  $ \lambda I $.
	Define $  \mu (L_a  + \lambda I ) $ to be the ring product $( \mu I ) (L_a+ \lambda I )$, then $S $ becomes an algebra over $K$ and $(a, \lambda ) \rightarrow L_a + \lambda I $ is an algebra homomorphism of $ R_1 $ onto $S$.
	Let $N$ be the kernel of this mapping and write $ \hat{R_1} = R_1 / N $ for quotient algebra. Denote the coset $(a, \lambda) + N $ by $[a, \lambda]$. Hence $[a, \lambda]$ is an equivalence class of $ (a, \lambda) $ under equivalence relation $(a, \lambda) \equiv  (b, \mu)$ if and only if $ ax + \lambda x = b x + \mu x, ~\forall ~x \in R $.
	
	The following result leads to the partial solution of Problem 1.
	
\begin{theorem}	\label{ps2T3} With above notations, we have the following. 
	\begin{enumerate}
		\item The mapping $ a\rightarrow ~\bar{a}= [a,0] $ is an algebra homomorphism of $R$ into $\hat{R_1}$.
		\item If $ L(R) = \{x \in R ~|~ xy=0,~\forall~ y\in R\}=\{0\}$ (that is if the involution of $R$ is proper) then the mapping $a \rightarrow \bar{a} $ is injective.
		\item If the involution of $R$ is proper then $[a,\lambda] =0$ if and only if $[a^* , \lambda^*] =0 $ and the formula $[a, \lambda]^* =[a^*,\lambda^*]$ defines unambiguously proper involution in $\hat{R_1}$. 
		\item If $R$ is a weakly Rickart $*$-ring   $a \in R$ and $e$ is the right projection of $a$ in $ R $ then $ \bar{e}$ is the right projection of $ \bar{a}$ in $ \hat{R_1}$.
	\end{enumerate} 
\end{theorem}
\begin{proof}	$(1)$ and $(2)$ are easy verification.\\
	$(3):$ Observe that $[a ,\lambda]=0  $ if and only if $(a,\lambda) + N=N $ if and only if $(a,\lambda) \in N $ 
if and only if $(L_a +  \lambda I ) x= 0, ~ \forall ~x \in R$ if and only if $ax+\lambda x=0,~\forall ~x \in R$. Therefore in order to show $[a^*,\lambda^*]=0$ whenever $[a,\lambda] =0$ it is enough to show $a^*x +\lambda^*x=0,~\forall x\in R$.
	Consider $(a^*x+\lambda^* x)^*(a^*x+\lambda^*x)=(x^*a + \lambda x^* )(a^* x + \lambda^* x ) $ 
	$=  x^* aa^* x +  x^*  a \lambda^*   x +  \lambda x^*  a^*  x +  \lambda  x^*  \lambda^* x $ \\
	$=  x^* \left\lbrace a (a^* x ) + \lambda( a^*  x) \right\rbrace + x^* \left\lbrace a  (\lambda^* x)  + \lambda (\lambda^*  x) \right\rbrace $
	$= x^* 0 + x^* 0 = 0,~\forall x\in R $. Therefore
	$ a^* x + \lambda^* x=0, ~\forall x\in R $. 
	Hence $ [a,\lambda~]^* = [a^*,\lambda^* ]$  defines an involution in $ \hat{R_1} $. Also, $[a,\lambda]^* [a,\lambda] =0 $ implies that
	$[a^*,\lambda^*] [a,\lambda] =0$. That is $[a^* a +  \lambda a^* + \lambda^* a ,\lambda^* \lambda ]=0 $. This gives
	$(a^* a + \lambda  a^* + \lambda^* a)x +\lambda^* \lambda x=0,~\forall x\in R$. Therefore 
	$a^* ax +  \lambda a^* x + \lambda^* ax + \lambda^* \lambda x=0,~\forall x\in R$.
	Also, $ (a x + \lambda x )^* (a x + \lambda x ) =  ( x^* a^* + \lambda^* x^* ) ( ax+  \lambda x) = x^* a^* ax+x^* a^* \lambda x + \lambda^* x^* ax +\lambda^* x^* \lambda x = x^* [a^* ax+ a^* \lambda x + \lambda^*  a x +\lambda^* \lambda x ]= x^* [a^* ax+ \lambda a^*  x + \lambda^*  ax +\lambda^* \lambda x ] = x^*  0 =0,~ \forall x\in R $.
That is $ ax + \lambda x= 0,~\forall x \in R $. This gives $[a,\lambda] =0 $. Hence the involution $*$ is proper.\\
  $(4):$ Let $R$ be a weakly Rickart  $*$-ring $a \in R$ and $ e=RP(a)$. Then $ae=a$ and $ay= 0$ implies that  $ey=0 $ for $y \in R$. 
	We prove that $\bar{e} = RP(\bar{a}) $. 
	Consider $\bar{a} \bar{e} =[a,0][e,0]=[ae,0]=[a,0] = \bar{a} $.
	Let $\bar{y} = [b,\mu ]$ and $\bar{a} \bar{y} = 0 $.
	Then $[a,0] [b,\mu ] = [ab+ \mu a,0] = 0 $. This gives $(ab+\mu  a) x=0,~\forall x\in R$. That is $a(bx+\mu x)=0,~\forall x\in R $. This implies that
	$(eb+\mu e) x=0 ~\forall x\in R$. Therefore $ [eb + \mu e, 0] =0$. That is $ [e, 0] [b, \mu]=0$. This gives $ \bar{e} \bar{y} =0$. Therefore $\bar{e} = RP (\bar{a})$. 
\end{proof}	

 The following theorem gives a more general partial solution to Problem 1, we give the solution in which we replace integral domain $K$ by any commutative ring.
 
\begin{theorem}	\label{ps2T4} 
	 Let $R$ be a weakly Rickart $*$-ring and $K$ be a commutative $*$-ring with unity such that $R$ is a $*$-algebra over $K$ satisfying condition $(3)$ of Theorem \ref{Th2}. Then $R$ can be embedded in a Rickart $*$-ring with preservation of right projections.
\end{theorem}	
\begin{proof} Let $\hat{R_1} = R_1 /N = \{[ a ,\lambda ] ~|~ ( a , \lambda ) \in R_1 \}$ and $ \hat{R_1} $ has unity element $u = [0, 1]$.
By Lemma \ref{ps2L4}, it is enough to show that every self-adjoint element of $ \hat{R_1} $ has the right projection.
	Let $[ a, \lambda ] \in \hat{R_1}~$ be a self-adjoint element. If $\lambda = 0$ then $ e = RP (a)$ and by  Theorem \ref{ps2T3} $\bar{e} = RP (\bar{a} )$. Suppose $\lambda \neq 0$. Then by Lemma \ref{ps2L3} there exists a largest projection $g$ such that $a g = - \lambda g$.
	Now we show that $RP ([ a , \lambda ]) = [-g , 1 ]$.
	Note that $[- g , 1 ]$ is a projection. Also, $[ a , \lambda ] [- g , 1 ] = [ - ag - \lambda g +a ,\lambda] = [ a , \lambda ]$.
	Moreover $[ a , \lambda ] [ b , \mu ] = 0$ if and only if $[ ab + \mu a + \lambda b , \lambda \mu ] = 0$ if and only if $a b x + \mu a x + \lambda b x + \lambda \mu x =0,~ \forall x \in R$ if and only if $a( b x + \mu x) + \lambda (b x + \mu x )=0, ~ \forall x \in R$ if and only if $( a + \lambda e_x ) (b x + \mu x ) = 0$ where $e_x = LP (b x + \mu x )$ if and only if $( a + \lambda e_x ) e_x = 0,~ \forall x \in R$ if and only if $a e_x = - \lambda e_x,~ \forall x \in R$. Since $g$ is the largest projection such that $a g = - \lambda g$. Therefore $e_x \leq g$. This gives $ e_x g = g e_x = e_x$. Therefore $[ a , \lambda ] [ b , \mu ] = 0$ if and only if $( e_x - g ) e_x = 0,~ \forall x \in R$ if and only if $( e_x - g ) ( b x + \mu x) = 0,~ \forall x \in R$ if and only if $ - g ( b x + \mu x) + e_x ( b x + \mu x ) = 0,~ \forall x \in R$ if and only if $- g b x - \mu g x + b x + \mu x =0,~ \forall x \in R$ if and only if $[ -g b - \mu g + b , \mu ] = 0$ if and only if $[ - g , 1 ] [ b , \mu ] = 0$. 
	Hence $\hat{R_1}$ is a Rickart $*$-ring.
\end{proof}

\section{Unitification of Weakly p.q.-Baer $*$-rings}  

  We recall the following examples of p.q.-Baer $*$-rings. This also shows how the class of p.q.-Baer $*$-rings is different than the class of Rickart $*$-rings.  
  \begin{example}  [{\cite[Exercise 10.2.24.4]{Bir5}}]
 	Let $A$ be a domain, $A_n=A$ for all $n=1,2, \cdots$, and $B$ be the ring of $(a_n)_{n=1}^{\infty} \in \prod_{n=1}^{\infty}A_n$ such that $a_n$ is eventually constant, which is a subring of $\prod_{n=1}^{\infty}A_n$. Take $R=M_n(B)$, where $n$ is an integer such that $n >1$. Let $*$ be the transpose involution of $R$. Then $R$ is a p.q.-Baer $*$-ring which is not quasi-Baer (hence not a quasi-Baer $*$-ring). Also, if $A$ is commutative which is not Pr$\ddot{u}$fer, then $R$ is not a Rickart $*$-ring.
 \end{example}   
 \begin{example} [{\cite[Exercise 10.2.24.5]{Bir5}}]
 	Let $R$ be a $*$-ring. If $R$ is a right (or left) p.q.-Baer ring and $*$ is semiproper, then $R$ is a p.q.-Baer $*$-ring. Hence, if $R$ is biregular and $*$ is semiproper, then $R$ is a p.q.-Baer $*$-ring.  
 \end{example} 
 \begin{example}  [{\cite[Example 1.7]{Anil}}] Let $ R={\displaystyle\left\{\begin{bmatrix}  a  & b \\  c &   d   \end{bmatrix} \in M_2(\mathbb Z)~|~
 		a \equiv d,~b \equiv 0,~ {\rm and} ~c \equiv 0 ~(mod~2)\right \}}$.
 	Consider  involution $*$ on $R$ as the transpose of the matrix.
 	In \cite[Example 2(1)]{Kim}, it is shown that $R$ is neither right p.p. nor left
 	p.p. (hence not a Rickart $*$-ring) but
 	$r_R(uR) = \{0\}=0R$ for any nonzero element $u \in
 	R$. Therefore $R$ is a p.q.-Baer $*$-ring.
 \end{example}
Recall the following result which gives the condition on $m$ and $n$ so that the matrix ring $M_n(\mathbb Z_m)$ is a Baer $*$-ring and hence a Rickart $*$-ring.
 \begin{corollary} [{\cite[Corollary 7]{Tha2}}]
 	\label{c031'} (i) $M_n(\mathbb Z_m)$ is a Baer
 	$*$-ring for $n \geq 2$ if and only if $n=2$ and $m$ is a square
 	free integer whose every prime factor is of form $4k+3$.\\
 	(ii) $\mathbb Z_m$ is a Baer $*$-ring if and only if $m$ is a
 	square free integer.
 \end{corollary}   
The following example shows that the right projections in a $*$-ring need not be central covers.
 \begin{example} [{\cite[Example 2.8]{Anil}}]
 	\label{ex01} Let $A=M_2$($\mathbb{Z}_3$), 
 	which is a Baer $*$-ring (hence a p.q.-Baer $*$-ring and a Rickart $*$-ring) with transpose
 	as an involution.  There is an element $x \in A$ such that
 	$RP(x)$ is not equal to $C(y)$ for any $y \in A$. 
 \end{example}       
 The following is a partial solution of the Problem 2 given in \cite{Anil}. 
  
  \begin{theorem}[{\cite[Theorem 4.6]{Anil}}]  A weakly p.q.-Baer $*$-ring $R$ can be embedded in a p.q.-Baer $*$-ring, provided there exists, a ring $K$ such that,
  	\begin{enumerate}
  	\item $K$ is an integral domain with involution,
  \item $R$ is a $*$-algebra over $K$,
  \item For any $\lambda \in K-\{0\}$ there exists a projection
  $e_{\lambda} \in R$ that is an upper bound for the central covers
  of the right annihilators of $\lambda$, that is, for $t \in R$, if
  $\lambda~t=0$ then $C(t) \leq e_{\lambda}$.
  \end{enumerate}
  \end{theorem}

Let $\tilde{R}$ denote the set of all projections in a $*$-ring $R$.
 In a weakly p.q. Baer $*$-ring, following is called the condition ($\beta$): For any $0\neq \lambda \in K, \exists ~ e_\lambda \in \tilde {R}$ such that $\lambda x = 0$ implies that $C(x) \leq e_\lambda$, where $K$ is a commutative $*$-ring with unity.

   \begin{lemma} \label{lm3.7}
   Let $R$ be weakly p.q. Baer $*$-ring which is a $*$-algebra over a commutative $*$-ring $K$ with unity satisfying condition ($\beta$). Then for any $a \in R$ and $0\neq \lambda \in K$ there exists a largest central projection $g$ such that $a g = \lambda g$.
   \end{lemma}
   \begin{proof}
   	On the similar line of Lemma \ref{ps2L3}.
   \end{proof}

  The following result leads to the solution of Problem 2.  

\begin{theorem}  \label{thm5}
   With notation as defined earlier
  \begin{enumerate}
  	\item   The mapping $a \rightarrow \bar{a} = [a, 0]$ is an algebra homomorphism of $R$ into $\hat{R_1}$.
  	\item If $L(R) = \{ x \in R :  x y = 0,~ \forall y \in R \} = \{0\} $  then the mapping $a \rightarrow \bar{a}$ is injective and we may regard $R$ as embedded in $\hat{R_1}$.
 	\item If the involution $*$ is semi-proper then $[a, \lambda] = 0$ if and only if $ [a^*, \lambda^*] = 0$. Hence
  $[a, \lambda]^* = [ a^*, \lambda^*]$  defines involution in $\hat{R_1}$.
  	\item If $R$ is weakly p.q. Baer $*$-ring, $ a \in R, C(a) = e $ then $ C(\bar{a}) = \bar{e} $ in $\hat{R_1}$.
\end{enumerate}
\end{theorem} 
\begin{proof}
  $(1)$ is trivial.\\
 $(2)$ To prove $\phi : R \rightarrow \hat{R_1}$ given by $\phi (a) =\bar{a}$ is injective. 
  Let $\phi(a)=\phi(b)$. Then $\bar{a} = \bar{b}$, that is $[a, 0] = [b, 0]$.
  This gives $ ax = bx, \forall  x \in R$. Therefore 
  $(a-b)x =0,~ \forall  x \in R$. This gives $a-b =0$. Hence $a = b$.\\
  $(3)$ Suppose $R$ has semi-proper involution, therefore for $a\in R$, $a^* Ra = 0$ implies that $a= 0$.
  Now, $[a, \lambda ] = 0$ if and only if $ax + \lambda x = 0,~\forall x \in R$. 
  Also, for any $r \in R$, $(x^* a + \lambda x^*) r ( a^*x + \lambda^* x)
  = x^* a r a^* x + x^* a r \lambda^* x + \lambda x^* r a^* x + \lambda x^* r \lambda^* x $\\ 
  $= x^* \{a (r a^* x) + \lambda (r a^*x) \} + x^* \{ a(r \lambda^* x) + \lambda( r \lambda^* x)\}= x^* 0 + x^* 0 = 0$. Therefore $[a, \lambda] = 0$ if an only if  $(x^* a + \lambda x^*) R (a^* x + \lambda^* x) = 0$ if and only if $(a^* x + \lambda^* x) = 0$ if and only if $[a^* \lambda^*] = 0$.
  Hence $[a, \lambda ]^*=[a^* \lambda^*]$  defines an involution in $\hat{R_1}$.\\
 $(4)$ Let $R$ be weakly p.q. Baer $*$-ring, $a \in R$ and $C(a) = e$.
  Consider $\bar{a} \bar{e} = [a, e] [e, 0] = [ae, 0] =[a, 0] = \bar{a}$.
  Also, $\bar{a} \hat{R_1} [b, \mu] = 0$ if and only if  $\bar{a}\bar{e} \hat{R_1} [b, \mu] = 0$ if and only if $\bar{a} \hat{R_1}\bar{e} [b, \mu] = 0$ if and only if $[a, 0] \hat{R_1} [e b + \mu e 0] = 0$ if and only if $[a, 0] [x, \lambda] [e b + \mu e, 0] = 0$ if and only if $[a (x + \lambda e) (e b + \mu e), 0] = 0$ if and only if 
  $a (x + \lambda e) (e b + \mu e) = 0$ if and only if $a R (e b + \mu e)=0$ if and only if $e(e b + \mu e)=0$ if and only if  $ e b + \mu e=0$ if and only if $(e b + \mu e)x=0,~\forall x\in R$ if and only if $[e b + \mu e,0]=0 $ if and only if $ [e,0][b,\mu]=0$. Therefore $C(\bar{a}) =\bar{e} $.
  \end{proof}
  	
  Now we give the more general partial solution to the Problem 2, in which we replace integral domain $K$ by any commutative ring with unity.
  \begin{theorem}
  Let $R$ be a weakly p.q. Baer $*$-ring and $K$ be a commutative $*$- ring with unity such that $R$ is a $*$-algebra over $K$ satisfying condition ($ \beta $). Then $R$ can be embedded in a p.q. Baer $*$-ring with preservation of central covers. 
  \end{theorem}
\begin{proof}
   Let $\hat{R_1}~ = R_1/N = \{[a, \lambda] ~|~ (a, \lambda) \in R_1 \}$. Note that  $u =[0, 1]$
 is a unity element of $\hat{R_1}$. 
  We show that $\hat{R_1}$ is p.q. Baer $*$-ring. 
  It is enough to show that for every element $x \in \hat{R_1}$ there exists a central projection $e \in \hat{R_1}$ such that: (1) $x e = x$, (2) $x \hat{R_1} y = 0$ if and only if $e y = 0$.
  Let $x = [a,\lambda] \in \hat{R_1}$.
  If $\lambda=0 $, let $  C(a) =e $. By Theorem \ref{thm5}, $C(\bar{a})=\bar{e}$. 
   Suppose $\lambda \neq 0$, then by Lemma \ref{lm3.7} there exists the largest central projection $ g $ such that $ ag=-\lambda g$. Clearly $[-g,1] $ is a central  projection. Also, $[a,\lambda][-g,1]=[-ag+a-\lambda g,\lambda]=[a,\lambda] $, that is $xe=x $ with $ e=[-g,1],x=[a,\lambda] $. 
  Suppose $ [a,\lambda] \hat{R_1}[b,\mu]=0$. Therefore $[a,\lambda][r,0][b,\mu]=0 $ for all $r \in R$. This gives $[arb+\lambda rb+\mu ar+\lambda \mu r,0]=0 $ for all $r \in R$. This implies $arbx+\lambda rbx+\mu arx+\lambda \mu rx=0$ for all $r, x\in R$. That is $ar(bx+\mu x)+\lambda r(bx+\mu x)=0$ for all $r,x\in R$. Therefore $(ar+\lambda re_x) (bx+\mu x)=0$, where $e_x=C(bx+\mu x)$. This gives $(a+\lambda e_x)r(bx+ \mu x) =0$ for all $r \in R$. That is $(a+ \lambda e_x) R(bx + \mu x)=0 $. Therefore $(a+ \lambda e_x)e_x=0$. Hence $ae_x=- \lambda e_x$. Since $ g $  is a largest central projection such that $ag=(-\lambda)g $, therefore $e_x\leq g $. 
  Therefore $(1-g)e_x=0$. This gives $(1-g)e_x(bx+ \mu x) =0$. Thus $(1-g)(bx+ \mu x)=0$ for all $x \in R$. Hence $bx+\mu x - gb x- \mu g x=0$ for all $x \in R$. Therefore $[-gb- \mu g + b, \mu]=0$, that is $[-g, 1][b, \mu]=0$. Hence $\hat{R_1}$ is a p.q.-Baer $*$-ring

 \end{proof}   
    
    Disclosure statement: The authors report there are no competing interests to declare.\\


\end{document}